\documentclass[10pt,a4paper]{article}

\usepackage{geometry}
\usepackage[T1]{fontenc}
\usepackage[utf8]{inputenc}
\usepackage{authblk}
\usepackage{amsmath}
\usepackage{amssymb}
\usepackage{amsfonts}
\usepackage{amsthm}
\usepackage{enumerate}
\usepackage{paralist}
\usepackage{amsfonts}
\usepackage{tensor} 
\usepackage{lmodern}
\usepackage{upgreek}
\usepackage{mathrsfs}
\usepackage{nicefrac}
\usepackage{verbatim}
\usepackage{amsmath}
\usepackage[all, cmtip]{xy}\usepackage{amssymb}
\usepackage{amsthm}
\usepackage{extpfeil}
\usepackage{bbm}
\usepackage{paralist}
\usepackage[colorlinks,linkcolor=blue,citecolor=gray,pagebackref=true,urlcolor=blue,pdftex]{hyperref}

\usepackage{titlesec}
\titleformat*{\section}{\Large\sffamily}

\usepackage{graphicx}

\usepackage[all]{xy}
\usepackage{tikz}
\usepackage{relsize}
\usepackage{comment}
\usepackage{mathtools}

\theoremstyle{plain}
\newtheorem{thm}{Theorem}

\newtheorem*{thm*}{Theorem}

\newtheorem{prop}[thm]{Proposition}
\newtheorem{lem}[thm]{Lemma}
\newtheorem{cor}[thm]{Corollary}

\newtheorem*{prb*}{Problem}

\theoremstyle{definition}
\newtheorem{rem}[thm]{Remark}
\newtheorem{df}[thm]{Definition}

\newtheorem{cons}[thm]{Construction}
\newtheorem{obs}[thm]{Observation}

\theoremstyle{remark}

\newtheorem{ex}[thm]{Example}

\newcommand\Defn[1]{\emph{\color{blue}#1}}
\newcommand{\cm}[1]{}
\newcommand\mc[1]{\mathcal{#1}}

\newcommand\wt[1]{\widetilde{#1}}

\newcommand{\Z}{\mathbb{Z}}

\newcommand\F{\mc{F}}
\renewcommand{\k}{\mathbbm{k}}

\newcommand\m{\mathrm{m}}
\newcommand\gen{\mathrm{u}}
\newcommand{\multibinom}[2]{
  \left(\!\middle(\genfrac{}{}{0pt}{}{#1}{#2}\middle)\!\right)
}
\newcommand{\binoms}[2]{
  \left(\!\genfrac{}{}{0pt}{}{#1}{#2}\!\right)
}

\newcommand\h{\mathrm{h}}

\newcommand\f{\mathrm{f}}
\newcommand\M{\mathrm{M}}

\makeatletter
\def\dual#1{\expandafter\dual@aux#1\@nil}
\def\dual@aux#1/#2\@nil{\begin{tabular}{@{}c@{}}#1\\#2\end{tabular}}
\makeatother

\title{Face numbers of sequentially Cohen-Macaulay complexes and Betti numbers of componentwise linear ideals}
\author[1]{Karim A.~Adiprasito\thanks{adiprasito@math.huji.ac.il}}
\author[2]{Anders Bj\"{o}rner\thanks{bjorner@kth.se}}
\author[3]{Afshin Goodarzi\thanks{afshingo@kth.se}}

\affil[1]{Einstein Institute of Mathematics, Hebrew University, Jerusalem,~Israel}
\affil[2]{Matematiska institutionen, 
Kungl.  Tekniska H\"ogskolan, Stockholm, Sweden}
\affil[3]{Institut f\"ur Mathematik, Freie Universit\"at, Berlin, Germany}

\begin{document}
\maketitle

%%%%%%%%%%%%%%%%%%%%%%%%%%%%%%%%%%%%%%%%%%%%%%%%%%%%%%%%%%%%%%
\begin{abstract}
A numerical characterization is given of the $h$-triangles of sequentially Cohen-Macaulay simplicial complexes.
This result determines the number of faces of various dimensions and codimensions that are possible 
in such a complex, generalizing the classical Macaulay-Stanley theorem to the nonpure case. Moreover, we characterize the possible Betti tables of componentwise linear ideals. A key tool in our investigation is a bijection between shifted multicomplexes 
of degree $\le d$ and shifted pure $(d-1)$-dimensional simplicial complexes.
\end{abstract}

%%%%%%%%%%%%%%%%%%%%%%%%%%%%%%%%%%%%%%%%%%%%%%%%%%%%%%%%%%%%%%

\section{Introduction}

The notion of sequentially Cohen--Macaulay complexes first arose in combinatorics: Motivated by questions concerning subspace arrangements, Bj\"orner \& Wachs introduced the notion of nonpure 
shellability~\cite{Bj-WI,Bj-WII}. Stanley then introduced the sequentially Cohen--Macaulay property in order to have a ring-theoretic analogue of nonpure shellability~\cite{StanleyGreen}. Schenzel independently defined the notion of sequentially Cohen--Macaulay modules (called by him Cohen--Macaulay filtered 
modules \cite{Schenzel99}), inspired by earlier work of Goto. In essence, a simplicial complex is sequentially Cohen--Macaulay if and only if it is naturally composed of a sequence of Cohen--Macaulay subcomplexes, namely the pure skeleta of the complex, graded by dimension. 
They come with an associated numerical invariant, the so-called ${h}$-triangle, which measures the 
face-numbers of each component according to a doubly-indexed grading. Just as the classical $h$-vector determines the numbers of faces of various
dimensions of a simplicial complex, the ${h}$-triangle determines the numbers of faces in each component 
of the complex.

Motivated by the Macaulay-Stanley theorem for Cohen-Macaulay complexes, which are always pure, Bj\"orner \& Wachs~\cite{Bj-WI} posed the problem to characterize the possible $h$-triangles of sequentially Cohen--Macaulay simplicial complexes. Via a connection that seems to have up to now been overlooked,
this problem is equivalent to characterizing the possible Betti tables of componentwise linear ideals, see for instance \cite[Theorem 2.3]{CHH}, \cite{KK} and ~\cite[Proposition 12]{HRW}. After some significant initial progress, due to Duval \cite{Duval96} and Aravoma, Herzog \& Hibi~\cite{AHH}, which reduced these two questions to combinatorial settings, some partial results on the second question were obtained by Crupi \& Utano~\cite{CrupiUtano} and Herzog, Sharifan \& Varbaro~ \cite{HSV}.
Part of the difficulty of this nonpure ``Macaulay problem'' is that, in contrast to the classical situation, it does not suffice to use a criterion that makes a decision by only
pairwise ``Macaulay type'' comparisons of entries in the $h$-triangle.

Our main objective in this paper is to give a numerical characterization of the possible $h$-triangles of sequentially Cohen--Macaulay complexes. The method that we use 
is based on a modification of a correspondence between shifted multicomplexes and pure shifted simplicial complexes, provided by Bj\"{o}rner, Frankl \& Stanley~\cite{BFS}. Finally, we also provide a characteristic-independent characterization of the possible Betti tables of componentwise linear ideals using our main result and an observation made by Herzog, Sharifan \& Varbaro~\cite{HSV}.

The paper is organized as follows. In Section~\ref{sec:Pre}, we present some basic definitions and derive some necessary relations on the face numbers of sequentially Cohen--Macaulay complexes. Section~\ref{sec:BFS} is devoted to our study of the Bj\"{o}rner, Frankl \& Stanley (BFS) bijection, which we examine via a new connection to lattice paths. The numerical characterization of the possible $h$-triangles of sequentially Cohen--Macaulay complexes is the subject of Section~\ref{sec:char}. Finally, in Section~\ref{sec:Betti} we present a numerical characterization of the possible Betti tables of componentwise linear ideals.

\section{Preliminaries}\label{sec:Pre}

\paragraph{Simplicial complexes.}

A family $\Delta$ of subsets of the set $[n]:=\{1,2,\ldots,n\}$ is called a \Defn{simplicial complex} on $[n]$, if $\Delta$ is closed under taking subsets, i.e. if $F\in\Delta$ and $F'\subseteq F$, then $F'\in\Delta$. The members $F$ of $\Delta$ are called \Defn{faces} of $\Delta$. The \Defn{facets} of $\Delta$ are the inclusion-wise maximal faces; the set of all facets of $\Delta$ is denoted by $\mathcal{F}(\Delta)$. The \Defn{dimension} $\dim F$ of a face $F$ is one less than its cardinality and the \Defn{dimension} of $\Delta$ is defined to be the maximal dimension of a face. A simplicial complex of dimension $d-1$ will be called a \Defn{$(d-1)$-complex}. A $(d-1)$-complex $\Delta$ is called \Defn{pure}, if each facet of $\Delta$ has dimension $d-1$. 

For a $(d-1)$-complex $\Delta$, let $\Delta^i:=\{\text{ $i$-dimensional faces of $\Delta$ }\}$ and let $f_i:=|\Delta^i|$. The vector $\f(\Delta)=(f_{-1},f_0,\ldots,f_{d-1})$ is called the \Defn{$f$-vector} of $\Delta$. The subcomplex $\Delta^{(i)}:=\bigcup_{j\leq i} \Delta^j$ is called the \Defn{$i$-skeleton} of $\Delta$. The \Defn{pure $i$-skeleton} $\Delta^{[i]}$ of $\Delta$ is the pure $i$-complex whose set of facets is the set of $i$-dimensional faces of $\Delta$, that is $\mathcal{F}(\Delta^{[i]})=\Delta^i$.  The \Defn{$h$-vector} $\h(\Delta)=(h_0,h_1,\ldots,h_d)$ of $\Delta$ is defined by 
\[\sum_{i=0}^d h_i y^i\ =\ \sum_{i=0}^d f_{i-1}(1-y)^{d-i}y^i.
\]
 For a $(d-1)$-complex $\Delta$, let  $\wt{h}_{i,j}=\wt{h}_{i,j}(\Delta)=h_j(\Delta^{[i-1]})$. Then the triangular integer array $\wt{\mathtt{h}}(\Delta)=(\wt{h}_{i,j})_{0\leq j\leq i\leq d}$  is called the \Defn{$\wt{h}$-triangle} of $\Delta$. Also, define $h_{i,j}$ by the relation
 \begin{eqnarray}\label{wth}
 h_{i,j}=\wt{h}_{i,j}-\sum_{\ell=0}^j\wt{h}_{i+1,\ell}.
\end{eqnarray}
The triangular integer array $\mathtt{h}(\Delta)=(h_{i,j})_{0\leq j\leq i\leq d}$ is called the \Defn{$h$-triangle} of $\Delta$. Note that our definition of the $h$-triangle 
is equivalent to the one presented in~\cite[Definition 3.1]{Bj-WI}.

Let $\k$ be an infinite field and $S=\k[x_1,x_2,\ldots,x_n]$ the polynomial ring over $n$ variables. For a simplicial complex $\Delta$ on $[n]$, let $I_\Delta$ be the \Defn{Stanley--Reisner ideal} of $\Delta$, that is the ideal 
\[I_\Delta:=\langle\left\{ x_{i_1}x_{i_2}\ldots x_{i_r}\ :\ \{i_1,i_2,\ldots,i_r\}\notin\Delta\right\}\rangle\]
of $S$. The quotient ring $\k[\Delta]:=S/I_\Delta$ is called the \Defn{face ring} of $\Delta$. The complex $\Delta$ is said to be \Defn{Cohen--Macaulay over $\k$}, if $\k[\Delta]$ is Cohen--Macaulay (see e.g.~\cite[page 273]{Herzog-Hibi}, for Cohen--Macaulay rings). A topological characterization of the Cohen--Macaulay complexes can be found in the book by Stanley~\cite{StanleyGreen}. The reference to the base field will usually be dropped and we simply say that $\Delta$ is Cohen--Macaulay, or CM for short. 

A $(d-1)$-complex $\Delta$ is said to be \Defn{sequentially Cohen--Macaulay}, 
or SCM for short, if the pure $i$-skeleton $\Delta^{[i]}$ of $\Delta$ is CM for all $i\leq d-1$. 

A simplicial complex $\Delta$ on $[n]$ is called \Defn{shifted} if for all integers $r$ and $s$ with $1\leq r<s\leq n$ and all faces $F$ of $\Delta$ such that $r\in F$ and $s\notin F$ one has $\left(F\setminus\{r\}\right)\cup\{s\}\in\Delta$. Recall that every shifted complex is (non-pure) shellable \cite{Bj-WII}, and a shifted complex is CM if and 
only if it is pure.

% ------------------------------------------------------------------------

\paragraph{Face numbers of CM complexes.}

Let $W_n=\{w_1,w_2,\ldots,w_n\}$ be a set of variables.
%totally ordered by $w_s<w_r$ for $s>r$. 
A \Defn{multicomplex} $\M$ on $V\subseteq W_n$ is a collection of monomials on $V$ that is closed under divisibility. A multicomplex $\M$ on $V$ is said to be \Defn{shifted}, if for all $x_r$ and $x_s$ in $V$ with $r<s$ and all monomials $\m$ in $\M$ divisible by $x_r$ one has that $x_s\cdot \left(\m/x_r\right)\in \M$.

Let $\M^i$ denote the set of monomials in $\M$ of degree $i$. The sequence $\f(\M)=(f_0,f_1,\ldots)$ is called the \Defn{$f$-vector} of $\M$, where $f_i=|\M^i|$. The numerical characterization of $f$-vectors of multicomplexes (due to Macaulay \cite{Macaulay}) can be seen as the historical starting point for a line of research
that this investigation is part of.

The \Defn{$\ell$-representation} of a positive integer $p$ is the unique way of writing 
\[ p\ =\ {a_\ell\choose \ell}+{a_{\ell-1}\choose \ell-1}+\ldots+{a_e\choose e},
\]
where $a_\ell> a_{\ell-1}>\ldots>a_e\geq e\geq 1$. Define
\[\partial^\ell(p)\ =\ {a_\ell-1\choose \ell-1}+{a_{\ell-1}-1\choose \ell-2}+\ldots+{a_e-1\choose e-1}.
\]
Also set $\partial^\ell(0)=0$ for all $\ell$. A vector $\f=(f_0,f_1,\ldots)$ of non-negative integers is called an \Defn{M-sequence} if $f_0=1$ and $\partial^\ell(f_\ell)\leq f_{\ell-1}$ for all $\ell$. 

A complete characterization of the $h$-vectors of CM complexes is achieved by combining the results by Macaulay~\cite{Macaulay} and Stanley~\cite{StanleyGreen,StanleyCMC}. With some additional information taken from
Bj\"{o}rner, Frankl \& Stanley~\cite{BFS} 
%and Kalai~\cite{Kalai} 
we get all parts of the following theorem.

\begin{thm}[Macaulay-Stanley Theorem] \label{MStheorem}
For an integer vector $\h=(h_0,h_1,\ldots,h_d)$ the following are equivalent:
\begin{compactenum}[\rm (a)]
\item $\h$ is the $h$-vector of a CM complex on $[n]$ of dimension $d-1$;
\item $\h$ is the $h$-vector of a pure shifted complex on $[n]$ of dimension $d-1$;
\item $\h$ is the $f$-vector of a multicomplex on $\{w_1,\ldots,w_{n-d}\}$;
\item $\h$ is the $f$-vector of a shifted multicomplex on $\{w_1,\ldots,w_{n-d}\}$;
\item $\h$ is an M-sequence with $h_1\leq n-d$. 
\end{compactenum}
\end{thm}

This is one of the early pinnacles of algebraic combinatorics. To understand why this theorem is so remarkable, notice for instance that the $h$-numbers of a Cohen-Macaulay complex (motivated by ring theory) generally do NOT have a direct combinatorial interpretation in that same simplicial complex \cite{Lickorish}. However, there is another simplicial complex 
(combinatorially motivated) and a multicomplex with the same $h$-numbers,
and in which they do have a simple interpretation.

%------------------------------------------------------------------------

\paragraph{Face numbers of SCM complexes: some necessary conditions.}

The $h$-triangle of a shifted complex (more generally, a shellable complex) has a combinatorial interpretation that we now recall, the reader may consult~\cite{Bj-WI,Bj-WII} for more information. For a shifted complex $\Delta$, reverse lexicographic order of the facets is a shelling order with restriction map $\mathcal{R}(F)=F \setminus \sigma(F)$, where $\sigma(F)$ is the longest segment $\{s,\ldots,n\}\subseteq F$ if $n\in F$ and is empty otherwise. In particular, one obtains that
\begin{eqnarray}\label{shiftedht}
h_{i,j}(\Delta)\ =\ \left|\{F\in\mathcal{F}(\Delta)\ :\ |F|=i\text{ $\&$ } |\sigma(F)|=i-j\}\right|.
\end{eqnarray}
Algebraic shifting is an operator on simplicial complexes that associates to a simplicial complex a shifted complex, preserving many interesting invariants of the complex.
We refer the reader to the article by Kalai~\cite{Kalai} or the book by Herzog \& Hibi~\cite{Herzog-Hibi} to see the precise definition and properties. It was shown by Duval~\cite{Duval96} that a complex is SCM if and only if algebraic shifting preserves its $h$-triangle (or, equivalently $\wt{h}$-triangle). In particular, the set of $h$-triangles of SCM complexes coincides with the set of $h$-triangles of shifted complexes. 

Putting these facts together we can deduce the following necessary conditions. 
\begin{prop}[{c.f.~\cite[Theorem 3.6]{Bj-WI}}]\label{proposed}
If a triangular integer array $\wt{\mathtt{h}}(\Delta)=(\wt{h}_{i,j})_{0\leq j\leq i\leq d}$ is the $\wt{h}$-triangle of a SCM complex, then 
\begin{compactenum}[\rm (a)]
\item every row $\wt{\h}^{[i]}:=(\wt{h}_{i,0},\wt{h}_{i,1},\ldots,\wt{h}_{i,i})$ is an M-sequence; and
\item $\wt{h}_{i,j}\geq \displaystyle\sum_{\ell\leq j}\wt{h}_{i+1,\ell}$.
\end{compactenum}

\end{prop}

These necessary conditions are, however, not sufficient, as the following example shows.
\begin{ex}\label{ex:nobd} The triangular integer array
\[ \wt{\mathtt{h}}\ =\ \begin{array}{ccccc}
1 & & & & \\
1 & 5& & & \\
1 & 4 &7 & &  \\
1 & 3 & 3 &4 &  \\
1 & 2 & 0 &0 &0  \\
\end{array}
\] 
satisfies the conditions in Proposition~\ref{proposed}.
However, there exists {\emph{no}} SCM complex with the given array as its 
$\wt{h}$-triangle. 
\enlargethispage{5mm}

To see this, assume the contrary and let $\Delta$ be a shifted complex with $\wt{\mathtt{h}}(\Delta)=\wt{\mathtt{h}}$. Let 
$X$ and $Y$ be the pure $3$- and $2$-skeleta of $\Delta$, respectively. It follows from equation~\eqref{shiftedht} that $X$ is obtained by taking 
three iterated cones from a disjoint union of $3$ points. Now, looking at 
$f$-vectors, it is clear that the underlying graph ($1$-skeleton) of $Y$ is the same as the 
underlying graph of $X$, which is a complete $4$-partite graph
$K_{3,1,1,1}$. Now, if we remove from $Y$ the smallest vertex in the shifted 
ordering, the underlying graph becomes $K_{3,1,1}$. 
However, this graph has only $3$ missing triangles, whereas $Y$ has $4$ homology facets (i.e., facets $F$ with $\mathcal{R}(F)=F$). 
Thus we get a contradiction. 
\end{ex}

\noindent \textbf{Remark}
Criterion (e) in Theorem \ref{MStheorem} shows that in order to decide whether $\h$ is the $h$-vector of a Cohen-Macaulay simplicial complex
% is via a ``boundary criterion'', by which we mean that 
it suffices to check a certain criterion for pairs of entries of $\h$. The answer is yes if and only if the answer is yes for every pair of consecutive entries. 

The same is not true for SCM complexes 
%(not even for the special case of relative complexes).
as shown by the necessity of condition (b) in Proposition \ref{proposed},
as well as by Example \ref{ex:nobd}.
% previous example, for which  every pair of entries in the previous example is %consitent with a SCM complex (which we leave as an exercise to the reader), but %that the whole triangle is not.
More than pairwise checks are needed here.

%%%%%%%%%%%%%%%%%%%%%%%%%%%%%%%%%%%%%%%%%%%%%%%%%%%%%%%%%%%%%%

\section{A Combinatorial Correspondence}\label{sec:BFS}

Correspondences between monomials and sets (or, between sets with repetition and sets without) are well-known in combinatorics. 
%They typically use some version of the idea of ``telescoping''. 
We are going to make crucial use of such a correspondence,
namely a more precise and elaborated  version of the 
bijection defined in  ~\cite{BFS}, see Remark \ref{rem:BFS}. 
%based on a common connection to lattice paths. 
We call  this the BFS correspondence.
It is conveniently explained in terms of lattice paths.

By a \Defn{lattice path} from $(0,0)$ to $(r,a)$ we mean a path restricted to east (E) and north (N) steps, each connecting two adjacent lattice points. Thus, a lattice path can be seen as a word $L=L_1,L_2,\ldots,L_{r+a}$ on the alphabet $\{N,E\}$ with the letter $N$ appearing exactly $a$ times. For two lattice paths $L$ and $L'$, let $L<L'$ mean that $L$ never goes above $L'$. The poset consisting of all lattice paths from $(0,0)$ to $(r,a)$ ordered by this partial order will be denoted by $\mathcal{L}_{r,a}$.

The lattice paths in $\mathcal{L}_{r,a}$ can be encoded in two natural ways: either by the position of the north steps, or by the number of north steps in each column. Thus, for $L\in\mathcal{L}_{r,a}$ let us define:

\begin{compactitem}[$\bullet$]
 \item $\nu(L)$ is the set of positions within $L$
 of its north steps, i.e.\ $\nu(L):=\{i\ :\ L_i=N\}$; 
 \item $\lambda(L)$ is the monomial $\prod\limits_{i=1}^r w_i^{\lambda_i(L)}$, where $\lambda_i(L)$ is the number of north steps of $L$
 coordinatized  as $(i-1, j) \rightarrow (i-1, j+1)$, for some $j$.
% along column $i$, meaning    
\end{compactitem}
 
 \begin{ex}
 Let $L=NEENENNEEEN$. Then 
 $\nu(L)=\{1, 4, 6, 7, 11\}$ and $\lambda(L)=w_1 w_3 w_4^2$, see Figure \ref{fig:lp}.

 \begin{figure}[htb]
\centering 
  \includegraphics[width=0.56	\linewidth]{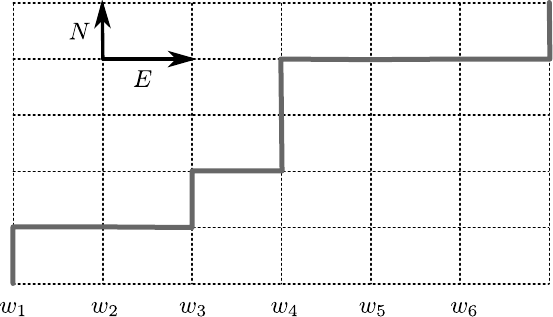}
\caption{The lattice path $L=NEENENNEEEN$ from $(0,0)$ to $(6,5)$.}
\label{fig:lp}
\end{figure}
 \end{ex}
 
A few more definitions are needed.  Recall that an {\em order ideal} $Q$ in a poset $P$ is a subset $Q\subseteq P$
 such that if $x\in Q$ and $z<x$ then $z\in Q$.
 We use the following notation:
 \begin{compactitem}[$\bullet$]
%\item $W_r =\{w_1, \ldots ,w_r\}$
\item  ${\binoms{[r+a]}{a}}=\mbox{the set of $a$-element  subsets of \{1, 2, \ldots, r+a\}}$ 
\item ${\multibinom{W_r}{\le a}}=\mbox{the set of monomials of degree $\le a$
in indeterminates $W_r =\{w_1, \ldots ,w_r\}$}$.
 \end{compactitem}

 We leave to the the reader to verify the following simple observations. 
\begin{prop} \label{prop-biject} The following holds true:
 \begin{compactenum} [\rm (a)]
 \item The map $\nu$ induces a bijection %$\ol{\nu}$ 
 between shifted set families 
 %${\binoms{[r+a]}{a}}_{\text{shifted}}$
 in ${\binoms{[r+a]}{a}}$ and order ideals 
 %$\ol{\partial}(\mathcal{L}_{r,a})$ 
 in $\mathcal{L}_{r,a}$. 
 \item The map $\lambda$ induces a bijection %$\ol{\lambda}$ 
 between shifted multicomplexes in ${\multibinom{W_r}{\le a}}$ 
 %on $W_r:=\{w_1,\ldots,w_r\}$ of degree less than or equal to $a$, 
 and order ideals 
 %$\ol{\partial}(\mathcal{L}_{r,a})$ 
 in $\mathcal{L}_{r,a}$. \qed
 \end{compactenum}
\end{prop}

Now, let $a$ be a positive integer and $\m$ a monomial on $W_r$ such that $\deg\m\leq a$. Define $\varphi^a(\m)$ to be the $a$-subset $\nu\lambda^{-1}(\m)$ of $[r+a]$. We drop the integer $a$ from the notation whenever there is no danger of confusion. Also, let $\psi$ be the inverse of $\varphi$. The situation is illustrated in
the following diagram of bijective maps:

 \[\xymatrix@C+4pt{
& \mathcal{L}_{r,a} \ar[dl]_{\lambda} \ar[d]^{\nu} \\
\multibinom{W_r}{\le a} \ar@/^/[r]|\varphi& \binoms{[r+a]_{\mbox{}}}{a} \ar@/^/[l]|\psi}
\]
 
 \vspace{3mm}

%\label{h/f-vectors}
\begin{prop} [BFS correspondence] \label{BFS_corr}\mbox{ }
\begin{compactenum}[\rm (a)]
\item\label{bfs1} The map $\varphi := \nu\lambda^{-1}$
induces a bijection $\overline{\varphi}$, with inverse $\overline{\psi}$, between shifted multicomplexes  in ${\multibinom{W_r}{\le a}}$
%on $W_r:=\{w_1,\ldots,w_r\}$ of degree $\le a$, 
%less than or equal to $a$,
and shifted set families  in ${ [r+a]\choose a}$.
\item\label{bfs2} For a pure shifted $(a-1)$-complex $\Delta$ with facets $\F(\Delta)$, one has $\h(\Delta)=\f(\overline{\psi} (\F(\Delta)))$.
\end{compactenum}
\end{prop}

%\begin{prop}
%\end{prop}

\begin{proof} The first part follows from Proposition \ref{prop-biject}.
For the second part, 
observe that for a facet $F$ of $\Delta$, the cardinality of its restriction 
$\mathcal{R}(F)$,
as discussed in connection with equation (\ref{shiftedht}),
 is equal to the number of $N$ steps in the last column of $\nu^{-1}(F)$. Hence, 
\[|\mathcal{R}(F)|\ =\ |F|-|\sigma(F)|\ =\ \deg \psi(F).\] 
This implies that $\h_i (\Delta)=\f_i (\overline{\psi} (\F(\Delta)))$ for all $i$.
\end{proof}

\begin{rem} \label{rem:BFS}
 %consider here, it can be checked that 
Our map $\varphi$ from monomials to sets 
can be shown to be identical to the map $\varphi$ defined in~\cite[page 30]{BFS}, up to relabeling (reversing the order of vertices and monomials).

It was shown in~\cite{BFS} for multicomplexes $\M$, that
$\text{$\M$ compressed $\Rightarrow$ $\overline{\varphi} (\M)$  shellable}. $
Also, we have seen here that 
$\text{$\M$  shifted $\Rightarrow$ $\overline{\varphi} (\M)$ shifted} $.
Since  compressed $\Rightarrow$ shifted $\Rightarrow$ 
 shellable, the second implication strengthens the previous one at both ends of the implication arrow.

  \end{rem}
 
\begin{df} Let $a$ be a positive integer and $\M$ a shifted multicomplex on $W_r$ of degree less than or equal to $a$. Define $\Phi^a(\M)=\Phi(\M)$ to be the simplicial complex whose set of facets, $\mathcal{F}(\Phi(\M))$, is
$\overline{\varphi}(\M)$. Also, for a pure shifted $(a-1)$-complex $\Delta$, set $\Psi(\Delta)$ to be the multicomplex consisting of monomials $\psi(F)$, for all facets $F$ of $\Delta$.
\end{df}

Let $a$ be a positive integer and $\M$ a shifted multicomplex on $W_r$ of degree less than or equal to~$a$. Define the \Defn{$a$-cone} $\mathscr{C}^a_{r+1}\M$ of $\M$ to be 
\[ \mathscr{C}^a_{r+1}\M\ =\ \left\{w_{r+1}^\ell\cdot\m\ :\ \m\in\M\text{ and }\deg\m+\ell<a\right\}.
\]
We will drop the indices $r+1$ and $a$ from the notation, when it is clear from the context. The cone construction on multicomplexes can be seen as a non-square-free analogue of the topological cone. However, it is more useful to see it as an analogue of yet another combinatorial construction; the codimension one skeleton of a simplicial complex.

\begin{prop}\label{codim-skeleton} 
Let $\M$ be a shifted multicomplex on $W_r$ of degree less than or equal to $a$. Then the set $\mathscr{C}^a_{r+1}\M$ is a shifted multicomplex on $W_{r+1}$. Furthermore, $\Phi^{a-1}(\mathscr{C}^a_{r+1}\M)$ is the $(a-2)$-skeleton of $\Phi^a(\M)$.
\end{prop}
\begin{proof}
Obviously, $\mathscr{C}^a_{r+1}\M$ is a pure shifted multicomplex on $W_{r+1}$ of degree $a-1$. Set $\Delta=\Phi^a(\M)$. Then the facets of 
the codimension one skeleton of $\Delta$ are 
\[\mathcal{F}(\Delta^{(a-2)})\ =\ \{F\setminus j\ :\ F\in\mathcal{F}(\Delta)\text{ \& } j\in\sigma(F)\}.\]
Let $F$ be a facet of $\Delta$ and $j$ an element in $\sigma(F)$. Observe that if $\nu^{-1}(F)=L_1,\ldots,L_{r+a}$, then $L_j=N$ and that $\nu^{-1}(F\setminus j)$ is the lattice path from $(0,0)$ to $(r+1,a-1)$ obtained by changing $L_j$ to an $E$ step. On the level of monomials this is the same as multiplying by a suitable power of $w_{r+1}$.
\end{proof}

We wish to extend the BFS bijection to the realm of not-necessarily-pure shifted complexes, the motivation being to make this useful tool 
available for SCM complexes. To do so we need the following definition.

\begin{df}[Metacomplex]\label{good} A
\Defn{$d$-metacomplex} is a sequence $\mathscr{M}=(\M^{[0]},\M^{[1]},\ldots,\M^{[d]})$ of multicomplexes on $W$ such that
%is said to be a \Defn{$d$-nested sequence of multicomplexes} on $W$ if 
 \begin{compactenum}[\rm (a)]
\item $\M^{[i]}$ is a 
%shifted 
multicomplex on $\{w_1,\ldots,w_{n-i}\}$ of degree less than or equal to $i$, for all $0\leq i\leq d$; and
\item $\mathscr{C}^i\M^{[i]}\subseteq\M^{[i-1]}$, for all $1\leq i\leq d$.
\end{compactenum}
Also, define the \Defn{$f$-triangle} of $\mathscr{M}$ to be the triangular integer array $\mathtt{f}(\mathscr{M})=(f_{i,j})_{0\leq j\leq i\leq d}$, where $f_{i,j}(\mathscr{M})$ is the number $f_j(\M^{[i]})$ of degree $j$ monomials in $\M^{[i]}$.
A metacomplex is \Defn{shifted} if all underlying multicomplexes $\M^{[i]}$ are.
\end{df}

For a $d$-metacomplex $\mathscr{M}$, let $\overline{\Phi}(\mathscr{M})$ be the union
\[\overline{\Phi}(\mathscr{M})=\bigcup_{i=0}^d \left\{\varphi^i(\m)\ :\ \m\in\M^{[i]}\right\}
\]
of subsets of $[n]$. It follows by Proposition~\ref{codim-skeleton} that the shadow of the collection of $i$-sets in $\overline{\Phi}(\mathscr{M})$ is contained in the collection of $(i-1)$-sets, for all $i\in [d]$. Thus, $\overline{\Phi}(\mathscr{M})$ is a shifted $(d-1)$-complex on $[n]$. 
Also, Proposition~\ref{BFS_corr}\eqref{bfs2}
implies that the $\wt{h}$-triangle of $\overline{\Phi}(\mathscr{M})$ coincides with the $f$-triangle of $\mathscr{M}$. 

Conversely, for a shifted $(d-1)$-complex $\Delta$ on $[n]$ the sequence
\[\overline{\Psi}(\Delta)\ :=\ (\Psi(\Delta^{[0]}), \Psi(\Delta^{[1]}),\ldots, \Psi(\Delta^{[d]}))\] 
is a metacomplex, whose $f$-triangle coincides with the $\wt{h}$-triangle of $\Delta$. Summarizing, we have established  this:

\begin{prop} [Extended BFS correspondence] %\label{BFS_corr}
\label{good-shifted} The pair $(\overline{\Phi},\overline{\Psi})$ is a bijection between shifted $d$-metacomplexes on $W$ and shifted $(d-1)$-complexes on $[n]$. Moreover, one has $\wt{\mathtt{h}}(\Delta)=\mathtt{f}(\overline{\Psi}(\Delta))$.\qed
\end{prop} 

The extended BFS correspondence can also be derived in terms
of lattice paths. Namely, let $\widehat{\mathcal{L}}_{r,a}$ be the set of 
all $\{N,E\}$ lattice paths beginning at $(0,0)$ and ending at some point
among $(n-j, j)$, $0\le j \le d$. 
Then the order ideals in $\widehat{\mathcal{L}}_{r,a}$ correspond
bijectively to shifted $d$-metacomplexes on $W$ on the one hand and to 
shifted $(d-1)$-complexes on $[n]$ on the other.

%%%%%%%%%%%%%%%%%%%%%%%%%%%%%%%%%%%%%%%%%%%%%%%%%%%%%%%%%%%%%%

\section{Face Numbers of SCM Complexes: A Numerical Characterization}\label{sec:char}

In this section we give a numerical characterization of possible $\wt{h}$-triangles of SCM complexes. For that purpose, we need to consider special kinds of integer systems $\mathcal{D}=\left\{q^\m\right\}_\m$, indexed by monomials $\m$ of degree less than or equal to $t$ on a set $W_s=\{w_1,w_2,\ldots,w_s\}$ of variables.

%The following definitions are needed.

\begin{df}\label{df:decom}
Let $s$ and $t$ be integers, $1\le s,t \le d$. An \Defn{$M_{s,t}$-array} is a
function $q: {\multibinom{W_s}{\le t}} \rightarrow \Z^+$
(whose values we write $q^\m$ rather than the conventional $q(\m)$),
%sending a monomial $\m$ to the integer $q^\m$, 
such that
\begin{compactenum}
[\rm (1)]
%\item[(1)]
\item
\label{Ax2} if $\deg(\m)=t-\ell$ and $m'=u_j\cdot (m/u_i)$ for some $i<j$ 
such that $u_i$ divides $\m$, then 
%$h_\ell\leq 
$q^\m\leq q^{\m'}$;
\item\label{Axi1} if $\deg\m=t$, then $q^\m=1$;
\item \label{Ax3} if $\m'= \m\cdot u_j$ for some $j\in [s]$ and $\deg(\m)=t-\ell$, then $\partial^\ell(q^\m)\leq q^{\m'}$. 
\end{compactenum}
%\end{df}
%\label{df:decom}
Furthermore, let $\h=(h_0,h_1,\ldots,h_d)$ be an $M$-sequence
and $r\in\Z^+$ an integer.
An \Defn{$M_{s,t}(\h)$-composition} of $r$ is
an $M_{s,t}$-array $\mathcal{D}=\{q^{\m}\}_\m$
such that 
\begin{compactenum}[(1)]
\setcounter{enumi}{+3}
\item\label{Ax2.1} $h_\ell\leq q^\m$ if $\deg(\m)=t-\ell$;
\item\label{Ax1} $\sum_\m q^\m=r$.
%and $m'=u_j\cdot (m/u_i)$ for some $i<j$  such that $u_i$ divides $\m$,
\end{compactenum}

For an  $M_{s,t}$-array $\mathcal{D}=\{q^{\m}\}_\m$, we let 
$\Sigma_{s}\mathcal{D}$ be the sum of all $q^\m$ such that $u_s$ divides $\m$. 
%For an  $M_{s,t}(h)$-composition $\mathcal{D}$ of $q$, we let 
%$\Sigma_{s}\mathcal{D}$ be the sum of all $q^\m$ such that $u_s$ divides $\m$. 
%We also set
That given, we define:
\begin{compactenum}[(1)]
\setcounter{enumi}{+5}
\item
$ \rho_{s,t}(r;\h)\ =\ \min\left\{ \Sigma_{s}\mathcal{D}\ :\ 
\mathcal{D}\ \text{is an $M_{s,t}(h)$-composition of $r$} \right\}.
$
\item
An  $M_{s,t}(\h)$-composition $\mathcal{D}$ of $r$ is said to be a \Defn{minimal} if $\Sigma_{s} \mathcal{D}=\rho_{s,t}(r;\h)$.
\end{compactenum}
\end{df}

\begin{ex} Let $\h=(1,4,9,4,1)$ and $r=22$. Then 
\begin{eqnarray*}
\mathcal{D}_1&=&\{q^\mathbf{1}=10, q^{u_1}=4, q^{u_2}=5, q^{u_1^2}=q^{u_2^2}=q^{u_1u_2}=1\},\quad\text{ and}\\
\mathcal{D}_2&=&\{q^\mathbf{1}=9,q^{u_1}=5,q^{u_2}=5,q^{u_1^2}=q^{u_2^2}=q^{u_1u_2}=1\}
\end{eqnarray*}
 are two minimal $M_{2,2}(\h)$-compositions of $22$. Whereas, 

\begin{eqnarray*}
\mathcal{D}_3&=&\{q^\mathbf{1}=9,q^{u_1}=4,q^{u_2}=6,q^{u_1^2}=q^{u_2^2}=q^{u_1u_2}=1\}
\end{eqnarray*}
is a non-minimal $M_{2,2}(\h)$-compositions of $22$.
\end{ex}

Clearly, for an integer $r$ 
and a triple $(\h,t,s)$, as in Definition~\ref{df:decom} and 
such that $r$ is greater than or equal to $\sum_{i=0}^t {s+i-1 \choose i} h_{t-i}$,
 an $M_{s,t}(\h)$-composition of $r$  exists. Hence, the quantity $\rho_{s,t}(r;\h)$ is well-defined. However, there is a canonical way to obtain a minimal composition that we now discuss.

\begin{rem}
Note that for $s=1$ condition~\eqref{Ax2}
is void and the array is linear. So, by conditions~\eqref{Axi1} and~\eqref{Ax3} the concept is then equivalent to that of an ordinary $M$-sequence.
\end{rem}

Let us first fix some notation. For a positive integer $p$ with $\ell$-representation
\[ p\ =\ {a_\ell\choose \ell}+{a_{\ell-1}\choose \ell-1}+\ldots+{a_e\choose e},
\]
where $a_\ell> a_{\ell-1}>\ldots>a_e\geq e\geq 1$. 
Define
\[\partial^{\tiny{\langle \ell,j\rangle}}(p)\ =\ {a_\ell-j\choose \ell-j}+{a_{\ell-1}-j\choose \ell-j-1}+\ldots+{a_e-j\choose e-j}.
\]
In particular, $\partial^{\tiny{\langle \ell,0\rangle}}(p)=p$ and $\partial^{\tiny{\langle \ell,1\rangle}}(p)=\partial^{\ell}(p)$. Note that $\partial^{\tiny{\langle \ell,j\rangle}}(p)$ is a lower bound for the number of monomials of degree $\ell-j$ in a multicomplex $\M$ with $f_\ell(\M)=p$.

Let us define a linear order on the monomials of degree less than or equal to $t$ on the set $U_s$ of variables. For all $i$, set $1<_iu_i<_iu_i^2<_i\ldots<_iu_i^t$. Finally set $<_{\tiny{\pi}}$ to be the product order of all $<_i$ induced by $u_1<\ldots<u_s$. Also, for a monomial $\m$ of degree $t-\ell$ on $U_s$ and a non-negative integer $j\leq t$ define
\[ \mathrm{c}_j(\m)\ =\ |\left\{\text{monomials $\m'$ on $U_{s,t-j}$ } \ :\ \deg\m'=t-j\text{ \& } \m<_{\tiny{\pi}}\m' \right\}|.
\]

\begin{cons} \label{minDecom}
Let $r$, $\h$, $t$  and $s$ be as in Definition~\ref{df:decom}. We construct a minimal $M_{s,t}(\h)$-composition of $r$ inductively as follows. 
\begin{compactenum}[\rm (1)]
\item Set $q^{\mathbf{1}}$ to be the maximum integer $p$ such that 
\[ \sum_{j=1}^t {s+j-1 \choose j}\cdot \max\{ h_{t-j},\partial^{\tiny{\langle t,j\rangle}}(p)\}\ \leq\ r-p.
\]
\item Let $\m$ be a monomial of a positive degree $t-\ell$ and assume that $q^{\m'}$ is defined for all monomials $\m'<_{\tiny{\pi}}\m$. Set $q^\m$ to be the maximum integer $p$ such that the quantity 
\[ \sum_{\m'<_{\tiny{\pi}} \m} q^{\m'}+\sum_{j=0}^\ell \mathrm{c}_{\ell-j}(\m)\cdot \max\{ h_{\ell-j},\partial^{\tiny{\langle \ell,j\rangle}}(p)\}\ +\sum_{j=\ell+1}^t \mathrm{c}_{j}(\m)\cdot \max\{ q^{\m'}\ :\ \deg\m'=t-j\text{ \& } \m'<_{\tiny{\pi}}\m \}
\]
is not greater than $r -p$. 
\end{compactenum}

\end{cons}
\noindent It is not difficult to see that the construction above yields a minimal
 $M_{s,t}(h)$- composition of  $r$.
%The minimality of this composition is rather obvious. 
This minimal composition will be called \Defn{the regular 
$M_{s,t}(h)$-composition of $r$}.

The following is our main result.
\begin{thm}\label{main}
A triangular integer array $\wt{\mathtt{h}}=(\wt{h}_{i,j})_{0\leq j\leq i\leq d}$ is the $\wt{h}$-triangle of a sequentially CM complex if and only if
 \begin{compactenum}[\rm (a)]
\item\label{partA} Every row $\h^{[i]}=(\wt{h}_{i,0},\wt{h}_{i,1},\ldots,\wt{h}_{i,i})$ is an M-sequence;
\item\label{partB} $\wt{h}_{i,j}\geq \sum_{\ell\leq j}\wt{h}_{i+1,\ell}$;
\item\label{partC} $\rho_{j,d-i}(\wt{h}_{i,j};\h^{[d]})\leq \wt{h}_{i,j-1}$.
\end{compactenum}
\end{thm}

\begin{proof}[Proof of the necessity part of Theorem~\ref{main}] The conditions~\eqref{partA} and~\eqref{partB} are already discussed in Proposition~\ref{proposed}. We shall prove the necessity of condition~\eqref{partC}.

Let $\Delta$ be a shifted $(d-1)$-complex and $\mathscr{M}:=(\M^{[0]},\M^{[1]},\ldots,\M^{[d]})$
its associated metacomplex on $W$. Let us also denote by $Q_{i,j}$ the set of all monomials in $\M^{[i]}$ of degree $j$. In particular, the cardinality of the set $Q_{i,j}$ is equal to $\wt{h}_{i,j}$. Now, for a monomial $\m$ on $U=\{w_{n-d+1},\ldots,w_{n-i}\}$, consider the set

\begin{eqnarray*}
Q_{i,j}^\m\ =\ \left\{\mathrm{p}=\mathrm{p}(w_1,\ldots,w_{n-i})\in Q_{i,j} \ :\ 
		\mathrm{p}(1,\ldots,1,w_{n-d},\ldots,w_{n-i})=\m
\right\}.
\end{eqnarray*}
We denote by $q^{\m}_{i,j}$ the cardinality of the set $Q_{i,j}^\m$. 
\medskip

 \noindent\textbf{Claim.}\textit{ Setting $u_t:=w_{n-d+t}$, for all $t\in[d-i]$, the system $\mathcal{D}=\{q^{\m}_{i,j}\}_\m$ is an $M_{j,d-i} (\h^{[d]})$-composition 
 of $\wt{h}_{i,j}$.}

\begin{proof}[Proof of  the claim] First observe that the sets $Q_{i,j}^\m$ form a partition of $Q_{i,j}$. Hence, the condition~\eqref{Ax1} of Definition~\ref{df:decom} is satisfied. Now, let $\m$ be a monomial of degree $j-\ell$ on $U$ and $\m'=u_k\left(\m/u_r\right)$ for some $r$ and $k$ such that $r<k\leq d-i$ and $u_r$ divides $\m$. It follows from Proposition~\ref{good-shifted} and Definition~\ref{good} that $\wt{h}_{d,\ell}\leq q^{\m}_{i,j}$. Also, since $\M^{[i]}$ is shifted for every $\mathrm{p}\in Q_{i,j}^\m$ one has $u_{k}\cdot u_{r}^{-1}\cdot \mathrm{p}\in Q_{i,j}^{\m'}$. Thus, we have $q^{\m}_{i,j}\leq q^{\m'}_{i,j}$ and the condition~\eqref{Ax2} is also valid. \\
Finally, set $\m'= u_k\cdot\m$ for some $k\leq d-i$. Let $\mathrm{p}$ be a monomial in $Q_{i,j}^\m$. For every $w$ in $\{w_1,\ldots,w_{n-d}\}$ that divides $\mathrm{p}$, the monomial $u_{k}\cdot  (\mathrm{p}/w)$ is in $Q_{i,j}^{\m'}$, since $\M^{[i]}$ is shifted. Hence, the shadow of the collection $\{\mathrm{p}/\m\ :\ \mathrm{p}\in Q_{i,j}^\m\}$ of monomials is contained in $\{\mathrm{p}'/\m'\ :\ \mathrm{p}'\in Q_{i,j}^{\m'}\}$. This verifies the condition~\eqref{Ax3} of Definition~\ref{df:decom}. Therefore, $\mathcal{D}=\{q^{\m}_{i,j}\}_\m$ is a $M_{j,d-i} (\h^{[d]})$-composition 
 of $\wt{h}_{i,j}$.
 
\end{proof}
\noindent To complete the proof of necessity, for every monomial $\m$ on $U$ that is divisible by $w_{n-i}$, set $\m'=\m/w_{n-i}$. The division map 
\[ {\times w_{n-i}^{-1}} \ : \ Q^\m_{i,j}\rightarrow Q^{\m'}_{i,j-1}
\]
is an injection, since $\M^{[i]}$ is a multicomplex. Hence, we have 
\[ \Sigma_{d-i}\mathcal{D}\ =\ \sum_{w_{n-i}\mid \m}q^{\m}_{i,j}\ \leq \ \sum_{ \m'}q^{\m'}_{i,j-1}  \ =\ \wt{h}_{i,j-1} .\]
Therefore we have $\rho_{j,d-i}(\wt{h}_{i,j};\h^{[d]})\leq \wt{h}_{i,j-1}$, as desired.
\end{proof}

\begin{proof}[Proof of the sufficiency part of Theorem~\ref{main}] Let $\wt{\mathtt{h}}$ be a triangular integer array satisfying Conditions~\eqref{partA},~\eqref{partB} and~\eqref{partC} of the statement. In the light of Proposition~\ref{good-shifted}, it suffices to construct a metacomplex $\mathscr{M}$ such that $\mathtt{f}(\mathscr{M})=\wt{\mathtt{h}}$. We construct $\mathscr{M}$ as follows:

\begin{compactenum}[\rm (1)]
\item Let $\M^{[d]}$ be the compressed multicomplex on $\{w_1,w_2,\ldots,w_{n-d}\}$ consisting of the first $\wt{h}_{d,j}$ monomials of degree $j$ in the reverse lexicographic order, for all $0\leq j\leq d$. 
\item Let $i$ be an integer less than $d$. We shall construct $\M^{[i]}$. For $j\leq i$, let $\mathcal{D}_{i,j}=\{q_{i,j}^\m\}_\m$ be the regular 
$M_{j,d-i} (\h^{[d]})$-composition of $\wt{h}_{i,j}$.
%composition of $\wt{h}_{i,j}$ with respect to $(\h,j,d-i)$. 
Consider the change of variables $u_t\rightarrow w_{n-d+t}$, for $t\in[d-i]$. We will use the same notation $\m$ to denote the image of $\m$ under this change of variables, this should not lead to any confusion. 

Now, for a monomial $\m$ of degree $\ell$ on $\{w_{n-d+1},\ldots,w_{n-i}\}$, let $P_{i,j}^\m$ be the set of the first $q_{i,j}^\m$ monomials of degree $j-\ell$ on $\{w_{1},\ldots,w_{n-d}\}$ in reverse lexicographic order. Also, let 
\[ Q_{i,j}\ =\ \bigcup_\m \left\{ \m\cdot\mathrm{p}\ :\ \mathrm{p}\in P_{i,j}^\m \right\}.
\]
Finally, we set $\M^{[i]} = \bigcup_{j=0}^i Q_{i,j}$. 
\end{compactenum}
Clearly, the number of elements of degree $j$ in $\M^{[i]}$ is $\wt{h}_{i,j}$. Also, given that the  $\M^{[i]}$'s are shifted multicomplexes, it follows from Condition~\eqref{partB} that $\mathscr{C}^{i+1}\M^{[i+1]}\subseteq\M^{[i]}$. Thus, it only remains to show that $\M^{[i]}$ is a shifted multicomplex. We first show that $Q_{i,j}$ is a shifted family of monomials, for all $j$.  Let $\mathrm{p}$ be a monomial in $Q^\m_{i,j}$, $w_r$ and $w_k$ two variables in $\{w_1,\ldots,w_{n-i}\}$ such that $k<r$ and $w_k$ divides $\mathrm{p}$. We shall show that $w_r\cdot(\mathrm{p}/w_k)\in Q_{i,j}$. Consider the following cases: 

\begin{compactdesc}
\item[\small{Case 1.} ($k<r\leq n-d$)] If $\mathrm{p}'=\mathrm{p}/\m$, then we have $w_r\cdot(\mathrm{p}'/w_k)\in P^\m_{i,j}$, since $P^\m_{i,j}$ is shifted. Hence, $w_r\cdot(\mathrm{p}/w_k)\in Q^\m_{i,j}$.
\item[\small{Case 2.} ($k\leq n-d<r\leq n-i$)] Note that the shadow of $P_{i,j}^\m$ is contained in $P_{i,j}^{\m'}$ by Condition~\ref{df:decom}\eqref{Ax3}, where $\m'=w_r\cdot\m$. Thus, $w_r\cdot(\mathrm{p}/w_k)\in Q^{\m'}_{i,j}$.
\item[\small{Case 3.} ($k\leq n-d<r\leq n-i$)]. Condition~\eqref{Ax2} of Definition~\ref{df:decom} implies that $P_{i,j}^\m$ is contained in $P_{i,j}^{\m'}$, for $\m'=w_r\cdot(\m/w_k)$. In particular,  $w_r\cdot(\mathrm{p}/w_k)\in Q^{\m'}_{i,j}$ and $Q_{i,j}$ is a shifted family. 
\end{compactdesc}
\noindent Finally, assume that $\mathrm{p}$ is a monomial in $Q_{i,j}$ and $w$ is a variable dividing $\mathrm{p}$. The shifted property insures that $w_{n-i}\cdot(\mathrm{p}/w)\in Q_{i,j}$. However, it follows from Condition~\eqref{partC} that $\mathrm{p}/w\in Q_{i,j-1}$. This shows that $\M^{[i]}$ is a multicomplex. 
\end{proof}

\section{Betti Tables of Componentwise Linear Ideals}\label{sec:Betti}
In this section we obtain a characterization of the possible Betti diagrams of componentwise linear ideals in a polynomial ring over a field of arbitrary characteristic. We start by recalling some definitions and refer the reader to the book by Herzog \& Hibi~\cite{Herzog-Hibi} for undefined terminology. 

A graded ideal $I$ is said to have an \Defn{$r$-linear resolution} if $b_{s,s+\ell}(I)=0$ for all $\ell\neq r$. 
For a graded ideal $I$, let $I_{( r)}$ be the ideal generated by all monomials of degree $r$ in $I$. Then $I$ is called \Defn{componentwise linear} if $I_{( r)}$ has an $r$-linear resolution for all~$r$. 

For square--free monomial ideals the notion of componentwise linearity is dual to sequential Cohen--Macaulayness in the sense that: the Stanley--Reisner ideal $I_\Delta$ of a complex $\Delta$ is componentwise linear if and only if its Alexander dual $\Delta^{\ast}$ is SCM. In particular, the Stanley--Reisner ideal of a shifted complex is componentwise linear; such an ideal is called \Defn{square--free strongly stable}. The square--free strongly stable ideals are square--free analogues of strongly stable ideals. Recall that, a monomial ideal $I\subseteq S$ is said to \Defn{strongly stable} if for every monomial $u$ in the minimal set $\mathcal{G}(I)$ of monomial generators of $I$ and all $i<j$ such that $x_j$ divides $u$, one has $x_i\cdot (u/x_j)$ is in $I$.

\begin{obs}[Herzog, Sharifan \& Varbaro, \cite{HSV}]
The set of Betti tables of componentwise linear ideals in a polynomial ring over a field of an arbitrary characteristic coincides with those of the strongly stable ideals. 
\end{obs}
In characteristic zero, it is known~\cite[Theorem 8.2.22]{Herzog-Hibi} that componentwise linearity can be characterized as ideals with stable Betti table under (reverse lexicographic) generic initial ideal. The interesting part of the observation is that the characterization of the Betti tables does not depend on the characteristic. We do not rewrite the observation here, instead we refer the reader to~\cite[page 1879]{HSV} for more details.

%The square-free version of Eliahou-Kervaire resolution~\cite[]{Peeva} provides a combinatorial formula for the graded Betti numbers of strongly stable ideals. 

%Let $R=\k[x_1\ldots,x_n]$ and $M_{< r}(R)$ be the set of monomials in $R$ of degree less than $r$. Now let $S=\k[x_1\ldots,x_{n+r}]$ and for every monomial $u=x_{i_1}x_{i_2}\ldots x_{i_t}$ in $M_{< r}(R)$, define $u^\sigma$ to be the monomial $x_{i_1}x_{i_2+1}\ldots x_{i_t+(t-1)}$ in $S$.
%Now let $I\subseteq R$ be a strongly stable ideal of regularity $r-1$. In particular every minimal generator of $I$ has degree less than $r$. Let $I^\sigma$ be the ideal in $S$ generated by $u^\sigma$ for all $u\in\mathcal{G}(I)$. It can be shown~\cite[Lemma 11.2.5]{Herzog-Hibi} that the ideal $I^\sigma$ is square-free strongly stable. Also, it follows immediately from the Eliahou--Kervaire formulas that the Betti tables of $I$ and $I^{\sigma}$ are the same (see~\cite[Lemma 11.2.6]{Herzog-Hibi}). Putting these all together, we may conclude the following. 

\begin{prop} The set of all Betti tables of $r$-regular componentwise linear ideals in  the polynomial ring $\k[x_1,\ldots,x_n]$ coincides with the set of all Betti tables of $r$-regular square--free strongly stable ideals in the polynomial ring $\k[x_1,\ldots,x_{n+r-1}]$.

\end{prop}
\begin{proof} 
Note that the Betti table of an ideal depends only on the set of generators, in the sense that if $I\subseteq\k[x_1,\ldots,x_n]$ and $J$ is the ideal generated by the set $\mathcal{G}(I)$ of generators of $I$ in the polynomial ring $\k[x_1,\ldots,x_{n+r}]$, then $I$ and $J$ have the same Betti tables. Now the conclusion follows from~\cite[Lemma 11.2.5]{Herzog-Hibi} and~\cite[Lemma 11.2.6]{Herzog-Hibi}.
\end{proof}

Let $I$ be a square--free strongly stable ideal in $S$. For $\gen\in\mathcal{G}(I)$, let us denote by $m(\gen)$ the biggest index $t$ such that $x_t$ divides $\gen$. If $d=\{\min\deg\gen\ :\ \gen\in\mathcal{G}(I)\}$, then for $\ell\geq d$ define:

\[
m_{k,\ell}(I)\ =\ |\left\{\gen\in \mathcal{G}(I)\ :\ \deg \gen= \ell\text{ \& }m(\gen)=k+\ell-1\right\}|.
\]
Clearly, $m_{k,\ell}=0$ if $k+\ell$ is greater than $n+1$. Thus we may think of the collection of doubly indexed $m$-numbers as a triangular array. The triangular integer array $\mathtt{m}(I)=:(m_{k,\ell})$; ${1\leq k\leq n-\ell+1\leq n-d+1}$, is called the \Defn{reduced array of generators} of $I$.

The square--free version of Eliahou--Kervaire implies (see~\cite[Subsection 7.2]{Herzog-Hibi}) that 
\[
b_{s,s+\ell}(I)\ =\ \sum_{k=s-1}^n \binoms{k}{s} m_{k,\ell}(I),
\]
or equivalently
\begin{equation}\label{BMN}
\sum_{s\geq 0} b_{s,s+\ell}(I)t^s\ =\ \sum_{s\geq 0} m_{s+1,\ell}(I)(1+t)^s.
\end{equation}

In particular, the characterisation of the possible Betti tables of square--free strongly stable ideals  is equivalent to characterizing the possible reduced arrays of generators. Following~\cite{HSV} for a square--free strongly stable ideal $I$ we also consider doubly indexed $\mu$-numbers defined recursively by the following relation
\begin{eqnarray}\label{mu-m}
m_{\ell,k}&=&\mu_{\ell,k}-\sum_{q=1}^\ell \mu_{q,k-1}.
\end{eqnarray}
The triangular integer array $\wt{\mathtt{m}}(I)=\left(\mu_{\ell,k} \right)$; ${1\leq k\leq n-\ell+1\leq n-d+1}$,  is called the \Defn{array of generators} of $I$.

The task of characterizing all possible (reduced) arrays of generators of square--free strongly stable ideals, however, translates nicely into combinatorics as follows.  

\begin{lem}\label{lem:shifted-mat}
Let $\Delta$ be a shifted simplicial complex on $[n]$. Then
\begin{eqnarray*}
m_{s+1,k}(I_\Delta)&=&h_{n-k,s}(\Delta^{\ast}).
\end{eqnarray*}
In particular, the array of generators of $I_\Delta$ is the same as $\wt{h}$-triangle of $\Delta^{\ast}$ (up to a suitable rotation).
\end{lem}
\begin{proof}
 For a facet $F$ in a shifted simplicial complex on $[n]$, let $\ell_F$ be the smallest integer such that $\ell_F\in\sigma(F)$, if $\sigma(F)$ is non-empty and otherwise set $\ell_F=n+1$. It follows from equation~\eqref{shiftedht} that
\begin{eqnarray*}
h_{n-k,s}(\Delta^{\ast})&=&|\left\{ F\in \mathcal{F}(\Delta^{\ast})\ :\  |F|=n-k\text{ \& } \ell_F=s+k+1 \right\}|.
\end{eqnarray*}
Now, observe that the complement map $F\mapsto F^c$ induces a bijection between $\mathcal{F}(\Delta^{\ast})$ and $\mathcal{G}(I_\Delta)$ with the property that: if $\gen$ is the image of $F$, then $\deg\gen+|F|=n$ and $\ell_F-1=m(\gen)$. Hence, we obtain 
\[
h_{n-k,s}(\Delta^{\ast})\ =\ |\left\{ \gen\in \mathcal{G}(I_\Delta)\ :\ \deg \gen=k\text{ \& } m(\gen)=s+k \right\}|=m_{s+1,k}(I_\Delta). 
\]
The last part of the statement follows by comparing equations~\eqref{wth} and~\eqref{mu-m}.
\end{proof}
The following corollary first appeared in~\cite[Proposition 12]{HRW}. Unfortunately, there is a misprint in the statement in the published version of that paper. 

\begin{cor} 
Let $\Delta$ be sequentially Cohen-Macaulay. Then
\[
\sum_{i\geq 0} b_{i,i+j}(I_{\Delta^{\ast}})t^i\ =\ \sum_{i\geq 0} h_{n-j-1,i}(\Delta)(1+t)^i.
\]
\end{cor}
\begin{proof}
Using algebraic shifting, it is enough to prove the result for the special case of shifted complexes. However, in this case the result follows from equation~\eqref{BMN} and Lemma~\ref{lem:shifted-mat}.
\end{proof}

\begin{thm}\label{MatGen}
A triangular integer array $\wt{\mu}=\left(\mu_{\ell,k}\right)$; ${1\leq k\leq n-\ell+r\leq n-d+r}$,  is the array of generators of an $r$-regular componentwise linear ideal with minimum degree of a generator equals to $d$ on $S$ if and only if 
 \begin{compactenum}[\rm (a)]
\item Every column $\mathrm{\mu}^{[j]}=(\wt{\mu}_{1,j},\wt{\mu}_{2,j},\ldots,\wt{\mu}_{n+r-j,j})$ is an M-sequence;
\item $\wt{\mu}_{i,j}\geq \sum_{\ell\leq i}\wt{\mu}_{\ell,j-1}$;
\item $\rho_{i,j-d+1}(\wt{\mu}_{i+1,j};\mathrm{\mu}^{[d]})\leq \wt{\mu}_{i,j}$.
\end{compactenum}
\end{thm}

\begin{proof}
This follows from Lemma~\ref{lem:shifted-mat} and Theorem~\ref{main}.
\end{proof}

%%%%%%%%%%%%%%%%%%%%%%%%%%%%%%%%%%%%%%%%%%%%%%%%%

\paragraph{Acknowledgements.}
K.~Adiprasito was supported by an EPDI/IPDE postdoctoral fellowship, and a Minerva fellowship of the Max Planck Society. A.~Bj\"orner and A.~Goodarzi
enjoyed the support of
Vetenskapsr\aa det, grant 2011-11677-88409-18.

%%%%%%%%%%%%%%%%%%%%%%%%%%%%%%%%%%%%%%%%%%%%%%%
{\small
\def\cprime{$'$}
\providecommand{\bysame}{\leavevmode\hbox to3em{\hrulefill}\thinspace}
\providecommand{\MR}{\relax\ifhmode\unskip\space\fi MR }
% \MRhref is called by the amsart/book/proc definition of \MR.
\providecommand{\MRhref}[2]{%
  \href{http://www.ams.org/mathscinet-getitem?mr=#1}{#2}
}
\providecommand{\href}[2]{#2}

}

\end{document}